\documentclass[12pt,reqno]{amsart}

\usepackage{amssymb,latexsym}
\usepackage[matrix,arrow]{xy}
\usepackage[usenames]{color}
\usepackage{slashed}
\usepackage{parskip}
\newtheorem{thm}{Theorem}[section]
\newtheorem{cor}[thm]{Corollary}
\newtheorem{lem}[thm]{Lemma}

\theoremstyle{definition}

\numberwithin{equation}{section}
 \theoremstyle{remark}
\newtheorem{rem}{Remark}[section]
\newcommand{\n}{\nabla}
\newcommand{\tn}{\tilde{\nabla}}

\newcommand{\ld}{\Lambda}
\newcommand{\intm}{\int_{M^n}}

\begin{document}
\title{\text Second Eigenvalue of Paneitz Operators and Mean Curvature\bf}
\author []{ Daguang Chen  and Haizhong Li}
\email { dgchen@math.tsinghua.edu.cn,hli@math.tsinghua.edu.cn}
\address{Department of
Mathematical Sciences, Tsinghua University, Beijing, 100084, P.R.
China}

\thanks { The second author was partly supported by NSFC grant No.10971110.}


\begin{abstract}
For $n\geq 7$, we give the optimal estimate for the second eigenvalue of Paneitz
operators for compact $n$-dimensional submanifolds in an $(n+p)$-dimensional space form.
\end{abstract}
\maketitle
\renewcommand{\sectionmark}[1]{}

\section{introduction}
Assume that $M^n$ is a compact Riemannian manifold immersed into Euclidean space $\Bbb
R^{n+p}$. In \cite{Reilly77}, Reilly  obtained the estimates for the  first eigenvalue
$\lambda_1$ of Laplacian
\begin{equation}\label{Reilly77}
 \lambda_1\leq \frac{n}{V(M^n)}\intm |H|^2,
\end{equation}
where $H$ is the mean curvature vector of immersion $M^n$ in $\Bbb R^{n+p}$, $V(M^n)$ is
the volume of $M^n$. In \cite{SI92}, El Soufi and Ilias obtained the corresponding
estimates for submanifolds in unit sphere $\Bbb S^{n+p}(1)$, hyperbolic space $\Bbb
H^{n+p}(-1)$ and some other ambient spaces. Motivated by \cite{LY82}, El Soufi and Ilias
\cite{SI00} obtained the sharp estimates for the second eigenvalue of Schr\"{o}dinger
operator for compact submanifolds $M^n$ in space form $\Bbb R^{n+p}$, $\Bbb S^{n+p}(1)$
and hyperbolic space $\Bbb H^{n+p}(-1)$.

Given a smooth 4-dimensional Riemannian manifold $(M^4,g)$, the Paneitz operator, discovered
in \cite{Paneitz83}, is the fourth-order operator defined by
\begin{equation*}
P^4f=\Delta^2f-\text{div}\Big(\frac23 R \ \text{Id}-2Ric\Big)df, \qquad \text{for}\quad
f\in C^\infty(M^4),
\end{equation*}
where $\Delta$ is the scalar Laplacian defined by $\Delta=\text{div} d$, $\text{div}$ is
the divergence with respect to $g$, $R, Ric$ are the scalar curvature and Ricci curvature
respectively. The Paneitz operator was generalized to higher dimensions by Branson
\cite{Branson}. Given  a smooth compact Riemannian n-manifold $(M^n,g)$, $n\geq5$, let
$P$ be the operator defined by (see also \cite{CHY2004})
\begin{equation}\label{Paneitz}
Pf=\Delta^2f-\text{div}\Big(a_nR\ \text{Id}+b_nRic\Big)df+\frac{n-4}{2}Qf,
\end{equation}
where
\begin{equation}\label{Qcurvature}
\begin{aligned}
  Q&=c_n|Ric|^2+d_nR^2-\frac{1}{2(n-1)}\Delta R\\
&=\frac{n^2-4}{8n(n-1)^2}R^2-\frac{2}{(n-2)^2}|E|^2-\frac1{2(n-1)}\Delta R,\qquad E=Ric-\frac Rn g,
\end{aligned}
\end{equation}
and
\begin{equation}\label{abcd}
\begin{aligned}
&a_n=\frac{(n-2)^2+4}{2(n-1)(n-2)},  & b_n=-\frac{4}{n-2}, \\
&c_n=-\frac{2}{(n-2)^2}  & d_n=\frac{n^3-4n^2+16n-16}{8(n-1)^2(n-2)^2}.
\end{aligned}
\end{equation}
The operator $P$ is also called Paneitz operator (or Branson-Paneitz operator).

In \cite{CHY2004,Gursky99,HR2001,XuYang2002}, the authors  investigated positivity of the
Paneitz operator. In analogy with the conformal volume in \cite{LY82},  Xu and Yang
\cite{XuYang2002} defined N-conformal energy for compact 4-dimensional Riemannian
manifold immersed in N-dimensional sphere $\Bbb S^N(1)$. In the same paper
\cite{XuYang2002}, the upper bound for the first eigenvalue of Paneitz operator was
bounded by using n-conformal energy. In \cite{ChenLi09}, we obtained the sharp estimates
for the first eigenvalue of Paneitz operator for compact 4-dimensional submanifolds in
Euclidean space and unit sphere.

The aim of this paper is to obtain the optimal  estimates for the second eigenvalue of
Paneitz operator in terms of the extrinsic geometry of the compact submanifold $M^n$ in
space form $R^{n+p}(c)$(the Euclidean space $\Bbb R^{n+p}$ for $c=0$, the Euclidean unit
sphere $\Bbb S^{n+p}(1)$ for $c=1$ and the hyperbolic space $\Bbb H^{n+p}(-1)$ for
$c=-1$). Considering the first eigenvalue $\ld_1$ of $P$, it is easy to  see that it is
bounded by the mean value of the Q-curvature  on $M$,
\begin{equation}\label{1eigenvalue}
\ld_1\leq \frac{n-4}{2V(M^n)}\intm Qdv_g.
\end{equation}
Moreover, the inequality  (\ref{1eigenvalue}) is strict unless $Q$  is constant.

For the second eigenvalue we have the following
\begin{thm}\label{main thm}
Let $\phi: M^n\to R^{n+p}(c)$ be an $n$-dimensional ($n\geq 7$) compact submanifold. Then
the second eigenvalue $\lambda_2$ of Paneitz operator satisfies
\begin{equation}\label{Reilly1}
\ld_2V(M^n)\leq \frac12n(n^2-4)\intm\Big(|H|^2+c\Big)^2dv_g+\frac{n-4}{2}\intm Qdv_g.
\end{equation}
Moreover, the equality holds if and only if $\phi(M^n)$ is an n-dimensional geodesic
sphere $\Bbb S^n(r_c)$  in $R^{n+p}(c)$, where
\begin{equation}\label{radius}
r_0=\frac12\left(\frac{n(n+4)(n^2-4)}{\ld_2}\right)^{1/4},\quad r_1=arcsin\  r_0,\quad r_{-1}=sinh^{-1}r_0.
\end{equation}
\end{thm}

\begin{rem}
For the n-dimensional geodesic sphere $\Bbb S^n(r_c)$ in $R^{n+p}(c)$, we have
\begin{equation*}
 \begin{aligned}
 \ld_1&=\frac1{16}n(n-4)(n^2-4)(|H|^2+c)^2\\
 \ld_2&=\frac1{16}n(n+4)(n^2-4)(|H|^2+c)^2\\
     Q&=\frac18n(n^2-4)\left(|H|^2+c\right)^2.
 \end{aligned}
 \end{equation*}
\end{rem}

From (\ref{Qcurvature}) and (\ref{Reilly1}), we can reach
\begin{cor}Under the same assumptions as in the theorem \ref{main thm}, then
\begin{equation}\label{Reilly2}
\ld_2V(M^n)\leq
\frac12n(n^2-4)\intm\Big(|H|^2+c\Big)^2dv_g+\frac{(n-4)(n^2-4)}{16n(n-1)^2}\intm R^2.
\end{equation}
Moreover, the equality holds if and only if $M^n$ is an n-dimensional geodesic sphere.
\end{cor}

\begin{rem}
 We note that our technique in proof of Theorem 1.1 does not work for $3\leq n\leq 6$,
so it is interesting to know that Theorem 1.1 is true or not for $3\leq n\leq 6$.
\end{rem}

\section{Some lemmas}
Assume that $\phi: M^n\to R^{n+p}(c)$ is an $n$-dimensional compact submanifold
in an $(n+p)$-dimensional space form $R^{n+p}(c)$. From \cite{LY82}(see also \cite{SI00}), it is known that
\begin{lem}\label{orth}
Let $w$ be the first eigenfunction of Paneitz operator $P$ on $M^n$.
Then there exists a regular conformal map
\begin{equation}\label{regularc}
\Gamma: R^{n+p}(c)\to \Bbb S^{n+p}(1)\subset R^{n+p+1}
\end{equation}
such that for all $1\leq \alpha\leq n+p+1$, the immersion $X=\Gamma\circ
\phi=(X^1,\cdots,X^{n+p+1})$ satisfies
\begin{equation*}
 \intm X^\alpha wdv_g=0
\end{equation*}
where $g$ is induced metric of $\phi: M^n\longrightarrow R^{n+p}(c)$.
\end{lem}

Assume that $\tilde {g}=e^{2u}g$ is a conformal transformation for $u\in C^\infty(M)$, then the scalar curvature obeys   \cite{RY1994}
\begin{equation}\label{cscalar}
e^{2u}{\tilde R}=R-2(n-1)\Delta u-(n-1)(n-2)|\nabla u|^2,
\end{equation}
the gradient operator and the Laplacian  follows
\begin{equation}\label{cLaplace}
{\tilde\Delta}f=e^{-2u}\left[\Delta f+(n-2) \n u \cdot\n f\right],
\end{equation}
\begin{equation}\label{cgradient}
\tn f=e^{-u}\n f,\qquad e^{2u}|{\tilde\nabla}f|^2=|\nabla f|^2,
\end{equation}
where $\n$ and $\Delta$ (resp. $\tn$ and $\tilde\Delta$) are the Levi-Civita connection
and Laplacian with respect to $g$ (resp. $\tilde g$).

We have the following relation under conformal transformation $\tilde {g}=e^{2u}g$ (see
page 766 in \cite{SI00}),
\begin{lem}
Let $\phi: M^n\to R^{n+p}(c)$ be an $n$-dimensional submanifold and $X=\Gamma\circ \phi$
as before. Then we have
\begin{equation}\label{p1}
e^{2u}(|\tilde {\textbf{h}}|^2-n|\tilde H|^2)=|\textbf{h}|^2-n|H|^2,
\end{equation}
where $\textbf{h}, \tilde {\textbf{h}}$ are the second fundamental form of the immersion
$\phi$ and $X$ respectively, $H=\frac1n\text{tr}{\textbf{h}}$ and $\tilde H=\frac1n
\text{tr}\tilde {\textbf{h}}$ are the mean curvature vectors, $u$ is defined by
\begin{equation}\label{p2}
e^{2u}=\frac{1}{n}|\nabla(\Gamma\circ\phi)|^2.
\end{equation}
\end{lem}

We need also the following result (see also \cite{SI00}):
\begin{lem}Let $\phi: M\to R^{n+p}(c)$ be an $n$-dimensional submanifold and
$X=\Gamma\circ \phi$ as before. Then we have
 \begin{equation}\label{p3}
e^{2u}\left(|{\tilde H}|^2+1\right)=|H|^2+c-\frac{2}{n}\Delta u-\frac{n-2}{n}|\nabla u|^2.
\end{equation}
\end{lem}
\begin{proof}
The Gauss equation for $\phi:M\to R^{n+p}(c)$ states
\begin{equation}\label{Gauss1}
|\textbf{h}|^2-n|H|^2=n(n-1)|H|^2+n(n-1)c-R,
\end{equation}
Similarly,
\begin{equation}\label{Gauss2}
|{\tilde {\textbf{h}}}|^2-n|{\tilde H}|^2=n(n-1)|{\tilde H}|^2+n(n-1)-{\tilde R},
\end{equation}
Combining (\ref{p1}), (\ref{Gauss1}) and (\ref{Gauss2}), we get
\begin{equation}\label{p4}
n(n-1)(|H|^2+c)-R=
\left[n(n-1)\left(|{\tilde H}|^2+1\right)-{\tilde R}\right]e^{2u},
\end{equation}
i.e.
\begin{equation}\label{p5}
 n(n-1)e^{2u}(|{\tilde H}|^2+1)=n(n-1)(|H|^2+c)-[R-e^{2u}{\tilde R}]
\end{equation}
From (\ref{cscalar}), we have
\begin{equation*}
R-e^{2u}{\tilde R}=2(n-1)\Delta u+(n-1)(n-2)|\nabla u|^2,
\end{equation*}
i.e.
\begin{equation}\label{p6}
\frac{1}{n(n-1)}(R-e^{2u}{\tilde R})=\frac{2}{n}\Delta u+\frac{n-2}{n}|\nabla u|^2.
\end{equation}
Inserting (\ref{p6}) into (\ref{p5}) yields (\ref{p3}).
\end{proof}

The following lemma is crucial in the proof of our Theorem 1.1.

\begin{lem}Let $\phi: M\to R^{n+p}(c)$ be an $n$-dimensional compact submanifold,
$X=\Gamma\circ \phi$ as before and $u$ be defined by (\ref{p2}), then
\begin{equation}\label{p7}
\intm e^{2u}(|H|^2+c)\leq
  \intm (|H|^2+c)^2-\frac{n-6}{n}\intm e^{2u}|\n u|^2.
\end{equation}
\end{lem}
\begin{proof}
Multiplying $e^{2u}$ in both sides of (\ref{p3}), we have
\begin{equation}\label{p8}
e^{4u}(|{\tilde H}|^2+1)=e^{2u}(|H|^2+c)-\frac{2}{n}e^{2u} \Delta
u-\frac{n-2}{n}e^{2u}|\nabla u|^2.
\end{equation}

Integrating (\ref{p8}) over $M$ and noting
$$
\intm e^{2u}\Delta u=-2\intm e^{2u}|\nabla u|^2,
$$
we can get
\begin{equation}\label{p9}
\intm e^{4u}\leq \intm e^{2u}(|H|^2+c)- \frac{n-6}{n}\intm
e^{2u}|\nabla u|^2.
\end{equation}
From the  Cauchy-Schwartz inequality and (\ref{p9}), we have
$$
\begin{aligned}
2\intm e^{2u}(|H|^2+c)&\leq \intm e^{4u}+\intm (|H|^2+c)^2\\
&\leq \intm e^{2u}(|H|^2+c)-\frac{n-6}{n}\intm e^{2u}|\nabla
u|^2+\intm (|H|^2+c)^2.
\end{aligned}
$$
This inequality implies (\ref{p7}).
\end{proof}

\section{Proof of Theorem 1.1 }

Assume that $\phi: M^n\to R^{n+p}(c)$ be an $n$-dimensional compact submanifold in an
$(n+p)$-dimensional space form $R^{n+p}(c)$. From Lemma \ref{orth}, there exists a
regular conformal map
\begin{equation*}
\Gamma: R^{n+p}(c)\to \Bbb S^{n+p}(1)\subset R^{n+p+1}
\end{equation*}
such that the immersion $X=\Gamma\circ
\phi=(X^1,\cdots,X^{n+p+1})$ satisfies
\begin{equation*}
 \intm X^\alpha wdv_g=0,\qquad \text{for all}\  1\leq \alpha\leq n+p+1,
\end{equation*}
where $w$ is the first eigenfunction of the Paneitz operator on $M^n$.

Let $\ld_2$ be  the second eigenvalue of Paneitz operator $P$. From the max-min principle
for the Paneitz operator, we have
 \begin{equation}\label{2eigenvalue}
\ld_2\intm (X^\alpha)^2dv_g\leq \intm P(X^\alpha)\cdot X^\alpha dv_g,\qquad 1\leq
\alpha\leq n+p+1.
\end{equation}

 Making summation over $\alpha$ from $1$ to $n+p$ in  (\ref{2eigenvalue}),
using the fact  $\sum_{\alpha=1}^{n+p+1}(X^\alpha)^2=1$ and (\ref{Paneitz}),  we can
obtain
\begin{equation}\label{p10}
\begin{aligned}
\ld_2V(M)&\leq
\sum_{\alpha=1}^{n+p+1}\intm P(X^\alpha)\cdot X^\alpha dv_g\\
&=
\intm\Big[\sum_{\alpha=1}^{n+p+1}\Delta^2X^\alpha\cdot X^\alpha-\sum_{j,k=1}^n<\{(a_nR\delta_{jk}+b_nR_{jk})X_j\}_k,X>\\
&\qquad+\frac{n-4}{2}\intm Q|X|^2\Big]dv_g\\
&=
\intm<\Delta X,\Delta X>dv_g+\intm\sum_{j,k=1}^n<(a_nR\delta_{jk}+b_nR_{jk})X_j,X_k>\\
&\qquad+\frac{n-4}{2}\intm Q|X|^2dv_g
\end{aligned}
\end{equation}
where we use Stokes' formula in the second equality.

By (\ref{cLaplace}) and (\ref{cgradient}), we have the following calculations
\begin{equation}\label{p11}
\begin{aligned}
&<\Delta X,\Delta X>\\
&=e^{4u}<{\tilde \Delta}X-(n-2){\tilde \nabla}u\cdot {\tilde\nabla}X,{\tilde \Delta}X-(n-2){\tilde \nabla}u\cdot {\tilde\nabla}X>\\
&=e^{4u}<n{\tilde H}-nX-(n-2){\tilde \nabla}u\cdot {\tilde\nabla}X,
n{\tilde H}-nX-(n-2){\tilde \nabla}u\cdot {\tilde\nabla}X>\\
&=e^{4u}[n^2|{\tilde H}|^2+n^2+(n-2)^2|{\tilde \nabla}u|^2]\\
&=e^{2u}[n^2e^{2u}|{\tilde H}|^2+n^2e^{2u}+(n-2)^2|\nabla u|^2],
\end{aligned}
\end{equation}
where ${\tilde H}$ is the mean curvature vector of $X=\Gamma\circ \phi:M^n\longrightarrow
\Bbb S^{n+p}(1)$, here we used  in the second equality  the following well-known formula
$\tilde \Delta X=n\tilde H-nX$.

Noting
\begin{equation}\label{p12}
<X_j,X_k>=e^{2u}\delta_{jk},
\end{equation}
and putting (\ref{p11}) into (\ref{p10}), we have
\begin{equation}\label{p13}
\begin{aligned}
\ld_2V(M)&\leq \intm e^{2u}
\Big[n^2e^{2u}\left(|{\tilde H}|^2+1\right)+(n-2)^2|\nabla u|^2\Big]dv_g\\
&\quad +(na_n+b_n)\intm Re^{2u}dv_g+\frac{n-4}{2}\intm Qdv_g.
\end{aligned}
\end{equation}

Putting (\ref{p3}) into (\ref{p13}) and by use of  the definitions of $a_n, b_n$ in
(\ref{abcd}), we obtain
\begin{equation}\label{p14}
\begin{aligned}
\ld_2V(M)&\leq
\intm e^{2u}\Big[n^2\big(|H|^2+c-\frac{2}{n}\Delta u-\frac{n-2}{n}|\nabla u|^2\big)\\
&\qquad+(n-2)^2|\nabla u|^2\Big]dv_g+
(na_n+b_n)\intm R e^{2u}dv_g+\frac{n-4}{2}\intm Qdv_g\\
&=
n^2\intm e^{2u}(|H|^2+c)dv_g
+(na_n+b_n)\intm R e^{2u}dv_g
\\ &\qquad+\frac{n-4}{2}\intm Qdv_g+\intm e^{2u}\Big((n-2)^2-(n-2)n+4n\Big)|\nabla u|^2dv_g\\
&=
n^2\intm e^{2u}(|H|^2+c)dv_g+2(n+2)\intm e^{2u}|\nabla u|^2dv_g\\
&\qquad +\frac{n^2-2n-4}{2(n-1)}\intm R e^{2u}dv_g+\frac{n-4}{2}\intm Qdv_g.
\end{aligned}
\end{equation}

From Gauss equation of $\phi: M^n\longrightarrow R^{n+p}(c)$
\begin{equation*}
R=n(n-1)c+n|H|^2-|\textbf{h}|^2
\end{equation*}
and $|\textbf{h}|^2\geq n|H|^2$, we have
\begin{equation}\label{p15}
R\leq n(n-1)(|H|^2+c).
\end{equation}

The equality holds in (\ref{p15}) if and only if $\phi: M^n\longrightarrow R^{n+p}(c)$ is
a total umbilical submanifold (see \cite{Li2002}).

By (\ref{p14}) and (\ref{p15}), we have
\begin{equation}\label{p16}
\begin{aligned}
\ld_2V(M)&\leq
\frac12n(n^2-4)\intm e^{2u}(|H|^2+c)dv_g+\frac{n-4}{2}\intm Qdv_g\\
&\qquad+2(n+2)\intm e^{2u}|\nabla u|^2dv_g.
\end{aligned}
\end{equation}

From (\ref{p7}), we have
\begin{equation}\label{p17}
\begin{aligned}
\ld_2V(M^n)&\leq
\frac12n(n^2-4)\intm\Big(|H|^2+c\Big)^2dv_g+\frac{n-4}{2}\intm Qdv_g\\
&\quad -\left(\frac12(n-6)(n^2-4)-2(n+2)\right) \intm
e^{2u}|\nabla u|^2dv_g\\
&=\frac12n(n^2-4)\intm\Big(|H|^2+c\Big)^2dv_g+\frac{n-4}{2}\intm Qdv_g\\
&\quad -\frac12(n+2)(n^2-8n+8)\intm e^{2u}|\nabla u|^2.
\end{aligned}
\end{equation}
Therefore, the inequality (\ref{Reilly1}) follows immediately from inequality (\ref{p17}) if $n\geq 7$.

If the equality holds in (\ref{Reilly1}), all the inequalities become equalities from
(\ref{2eigenvalue}) to (\ref{p17}). From (\ref{p17}), we can get $\n u=0$, i.e.
$u=constant$. In this case, (\ref{p9}) becomes equality, and then we can infer $\tilde
H=0$. (\ref{p3}) imply
\begin{equation}\label{Hcont}
|H|^2+c=e^{2u}=constant.
\end{equation}

Equality case in (\ref{p15}) give us $|\textbf{h}|^2=n|H|^2$, that is,
\begin{equation}\label{umb}
h_{ij}^\alpha=H^\alpha\delta_{ij},
\end{equation}
i.e., $\phi(M^n)$ is a totally umbilical submanifold in $R^{n+p}(c)$ (in \cite{SI00},
also called a geodesic sphere).

From (\ref{umb})and  Gauss equation of $\phi$,  we have
\begin{equation}\label{Gauss3}
\begin{aligned}
R_{ijkl}&=c(\delta_{ik}\delta_{jl}-\delta_{il}\delta_{jk})+h^\alpha_{ik}h^\alpha_{jl}-h^\alpha_{il}h^\alpha_{jk}\\
&=c(\delta_{ik}\delta_{jl}-\delta_{il}\delta_{jk})+H^\alpha
H^\alpha \delta_{ik}\delta_{jl}-H^\alpha H^\alpha  \delta_{il}\delta_{jk}\\
&=(|H|^2+c)(\delta_{ik}\delta_{jl}-\delta_{il}\delta_{jk}),\\
R_{ij}&=(n-1)(|H|^2+c)\delta_{ij},\\
R&=n(n-1)(|H|^2+c).
\end{aligned}
\end{equation}
where $R_{ijkl}, R_{ij}$ and $R$ are the components of Riemannian curvature tensor, the
Ricci tensor and scalar curvature of $M^n$, respectively.

By the definition of $Q$-curvature in (\ref{Qcurvature}), we have by use of (\ref{abcd})
\begin{equation}\label{Qcurvature1}
\begin{aligned}
Q&=-\frac{2}{(n-2)^2}|Ric|^2+\frac{n^3-4n^2+16n-16}{8(n-1)^2(n-2)^2}R^2\\
&=-\frac{2}{(n-2)^2}(n-1)^2(|H|^2+c)^2\delta_{ik}\delta_{ik}\\
&\qquad+\frac{n^3-4n^2+16n-16}{8(n-1)^2(n-2)^2}
n^2(n-1)^2(|H|^2+c)^2\\
&=\frac18n(n^2-4)\left(|H|^2+c\right)^2.
\end{aligned}
\end{equation}

Therefore, for the equality case in (1.6), we have
\begin{equation}\label{ld2}
\ld_2=\frac{1}{16}n(n+4)(n^2-4)(|H|^2+c)^2.
\end{equation}
From (\ref{ld2}), we have
\begin{equation}\label{mean}
|H|^2+c=4\sqrt{\frac{\ld_2}{n(n+4)(n^2-4)}}.
\end{equation}
Therefore, from (\ref{Gauss3}) and (\ref{mean}), we deduce that $\phi(M)$ is a geodesic
sphere $\Bbb S^n(r_c)$ with radius $r_c$ defined by (\ref{radius}).

Conversely, suppose that  $\phi(M)$ is a geodesic sphere $\Bbb S^m(r_c)$ with radius $r_c$ defined by (\ref{radius}) in space form $R^{n+p}(c)$. It is  easily deduced that the section curvature
\begin{equation}\label{sectionc}
R_{ijij}=4\sqrt{\frac{\ld_2}{n(n+4)(n^2-4)}},\quad i\not=j.
\end{equation}
From (\ref{Gauss3}), we obtain (\ref{mean}). Therefore the equality holds in
(\ref{Reilly1}). We complete the proof of Theorem 1.1.

\begin{rem} If we assume that the scalar curvature $R$ is
nonnegative, from (\ref{p15}) we have
\begin{equation}\label{rsq}
  R^2\leq n^2(n-1)^2(|H|^2+c)^2.
\end{equation}
Inserting (\ref{rsq}) into (\ref{Reilly2}), we have under $R\geq 0$ and the same
assumptions as in the theorem \ref{main thm}
\begin{equation}\label{Reilly3}
\ld_2V(M)\leq \frac{1}{16}n(n+4)(n^2-4)\intm(|H|^2+c)^2.
\end{equation}
Moreover, the equality holds if and only if $M^n$ is an n-dimensional geodesic sphere.
\end{rem}

\end{document}